\documentclass[preprint,12pt,numbers]{elsarticle}
\usepackage[english]{babel}

\usepackage[letterpaper,top=2cm,bottom=2cm,left=3cm,right=3cm,marginparwidth=1.75cm]{geometry}

\usepackage{graphicx}%
\usepackage{multirow}%
\usepackage{amsmath,amssymb,amsfonts}%
\usepackage{amsthm}%
\usepackage{mathrsfs}%
\usepackage[title]{appendix}%
\usepackage{xcolor}%
\usepackage{textcomp}%
\usepackage{arydshln}
\usepackage{manyfoot}%
\usepackage{tikz}
\usepackage{booktabs}%
\usepackage{algorithm}%
\usepackage{algorithmicx}%
\usepackage{algpseudocode}%
\usepackage{listings}%
\usetikzlibrary{quantikz2}
\usepackage{physics}
\usepackage[colorlinks=true, allcolors=blue]{hyperref}
\newcommand*\circledc{%
  \mathbin{
    \tikz[baseline=-0.75ex]{ 
      \node[
        shape=circle,
        draw,
        inner sep=0pt,      
        minimum size=1.5ex, 
        line width=0.4pt,
        anchor=center       
      ] (char) {\fontsize{6.5}{6.5}\selectfont c};
    }
  }
}

\newcommand{\diag}{\operatorname{diag}}
\newtheorem{theorem}{Theorem}
\usepackage{lineno}
\newtheorem{example}{Example}%
\newtheorem{remark}{Remark}%
\newtheorem{lemma}[theorem]{Lemma}
\newtheorem{definition}{Definition}%
\newtheorem{proposition}{Proposition}
\newtheorem{coro}{Corollary}

\journal{.}
\begin{document}

\begin{frontmatter}

\title{On Properties of the Conformal Circulant Matrices}
\author{Cristina Manzaneda \corref{cor1}}
\address{Departamento de Matem\'{a}ticas, Universidad Cat\'{o}lica del Norte. Antofagasta, Chile}
\ead{cmanzaneda@ucn.cl}
\author{Enide Andrade }
\address{Center for Research and Development in Mathematics and Applications,
 Department of Mathematics,  University of Aveiro, Portugal.}
\ead{enide@ua.pt}
\cortext[cor1]{Corresponding author}

\begin{abstract}
In this paper, we introduce and study conformal permutation matrices. In particular, for a permutation $\pi_s\in S_n$, we define the corresponding conformal permutation matrix $H_{\pi_s}$ and show that its $d$th power, where $d=n/\gcd(s,n)$, is a diagonal block matrix whose blocks are products of certain matrix blocks $h_i$. We further show that these diagonal blocks share a common subset of eigenvalues and prove that $H_{\pi_s}$ is block-diagonalizable. We also introduce conformal circulant matrices and investigate their eigenvalues. Finally, we establish a necessary and sufficient condition for a rectangular matrix to be conformal circulant.
\end{abstract}

\end{frontmatter}

\section{Introduction}
\noindent The study of partitioned matrices is fundamentally rooted in the necessity of managing linear systems whose dimensions exceed the capacity of traditional scalar processing. Let $M_{m,n}(\mathbb{K})$ be the set $m\times n$ matrices with entries on the field $\mathbb{K}$. When $m=n$ this set is denoted by $M_{n}(\mathbb{K}).$ As established by \cite{Horn}, the partitioning of a matrix $A \in M_{m,n}(\mathbb{K})$,  provides a formal framework to reveal underlying structures that would otherwise remain hidden. Additionally, $F=\frac{1}{\sqrt{n}}[\omega^{(i-1)(j-1)}]$, $1\leq i,j\leq n,$
is the discrete Fourier transform, where $\omega=\mbox{\rm{exp}}(\frac{2\pi  \mathrm{i}}{n})$ is a primitive $n$-th root of unity, with $\mbox{\rm{i}}^2=\sqrt{-1}$. In this context, the analysis of rectangular blocks of orders $m_i \neq n_j$ is of particular relevance, as it offers the essential flexibility required to model heterogeneous systems commonly found in fields such as particle physics, structural engineering, and High-Performance Computing (HPC).

A partition of a matrix consists of dividing the matrix into smaller submatrices (blocks) in such a way that every entry in the original matrix belongs to exactly one submatrix.

Partitioning of matrices is often a convenient device for the perception of the useful structure. For example, partitioning $B = [b_1 \cdots b_n] \in M_n(\mathbb{K})$ according to its columns $b_i, i=1, \ldots, n$ reveals the presentation $AB = [Ab_1 \cdots Ab_n]$ of the matrix product, partitioned according to the columns of $AB$.
\medskip 

A partition of a set $\mathcal{S}$ is a collection of subsets of $\mathcal{S}$ such that each element of $\mathcal{S}$ is a member of one and only one of the subsets. For example, a partition of the set $\{1, 2, \dots, n\}$ is a collection of subsets $\alpha_1, \dots, \alpha_t$ (called index sets) such that each integer between 1 and $n$ is in one and only one of the index sets. A sequential partition of $\{1, 2, \dots, n\}$ is a partition in which the index sets have the special form $\alpha_1 = \{1, \dots, i_1\}, \alpha_2 = \{i_1 + 1, \dots, i_2\}, \dots, \alpha_t = \{i_{t-1} + 1, \dots, n\}$.
\medskip

Let $A \in M_{m,n}(\mathbb{K})$. For index sets $\alpha \subseteq \{1, \dots, m\}$ and $\beta \subseteq \{1, \dots, n\}$, we denote by $A[\alpha, \beta]$ the submatrix of $A$ with entries that lie in the rows and columns of $A$ indexed by $\alpha$ and $\beta$, respectively. If $\alpha = \beta$, the submatrix $A[\alpha] = A[\alpha, \alpha]$ is a principal submatrix of $A$. An $n\times n$ matrix has $\binom{n}{k}$ distinct principal submatrices of size $k$.

\begin{example}
   $$\begin{bmatrix} 1 & 2 & 3 \\ 4 & 5 & 6 \\ 7 & 8 & 9 \end{bmatrix} [\{1, 3\}, \{1, 2, 3\}] = \begin{bmatrix} 1 & 2 & 3 \\ 7 & 8 & 9 \end{bmatrix}.$$ \hfill{$\diamond$}
\end{example}

\medskip 

If a matrix is partitioned by sequential partitions of its rows and columns, the resulting partitioned matrix is called a block matrix.

\medskip
It is worth noting that most of the fundamental results concerning structural matrix theory have traditionally been developed for matrices partitioned into square blocks. However, the study of matrices partitioned into rectangular blocks is equally important, as such structures arise naturally in several contemporary applications. In particular, rectangular block matrices play a significant role in quantum data analysis  where the underlying data often exhibit heterogeneous dimensions and require more general block configurations. 

Rectangular matrices appear naturally in problems of quantum data analysis, matrix classification, and dimensionality reduction \cite{Duan17,Duan19}. Furthermore, the study of block circulant matrices is used in block encoding techniques in modern quantum algorithms \cite{Camps, Sunderhauf}. Although circulant matrices and other structured matrices, such as Toeplitz, Hankel, and block circulant matrices, have been thoroughly studied, a unified framework for studying block rectangular circulant matrices within the context of block encoding does not appear to exist. Consequently, the development of a theory and spectral study for rectangular circulant matrices could provide new tools for quantum data processing and the efficient simulation of physical systems.

This motivates the investigation of theoretical properties and results in the broa\-der setting of rectangular block matrices.

\medskip
The main contributions of this paper can be summarized as follows:
\begin{itemize}
    \item In Section 2 some classical partition definitions are reviewed and the definition of sub-order conformity is introduced. Then the conformal product of conformal partitioned matrices is introduced and some operations and properties are described. 
    \item In Section 3, for a certain permutation $\pi_s \in S_n$, the conformal permutation matrix, $H_{\pi_s}$, is introduced. The power $d$ of $H_{\pi_s}$, where $d = \frac{n}{\gcd(s,n)}$, is then expressed as a diagonal block matrix whose blocks $D_{i}, i=1, \ldots, n$ are products of certain matrix blocks $h_i$ of order $m_i \times m_{\pi_s(i)}$. Then it is proven that these matrices $D_i$ will share an identical subset of eigenvalues. Moreover, under certain conditions, it has also been proven that $H_{\pi_s}$ is a block-diagonalizable matrix.
    \item In Section 4 the conformal circulant matrix is introduced and its eigenvalues are investigated. A necessary and sufficient condition for a rectangular matrix to be conformal circulant is presented and their spectral properties are studied.
\end{itemize}

\smallskip
{\bf Notation:} The greatest common divisor of two integers $a$ and $b$ is written as $\gcd(a,b)$. We write $I_n$ for the identity matrix of order $n$. For a square matrix $A$, its inverse is written as $A^{-1}$ and its spectrum by $\sigma{(A)}$.For any matrix $A$, $A^T$ represents the transpose of a matrix $A$. Moreover, $\mathbf{1}=(1,1,\ldots,1)$. If $A$ is a square matrix and $\lambda$ is an eigenvalue of $A$ with corresponding eigenvector $u$, then
$(\lambda, u)$ is called an eigenpair of $A$. The direct sum of two matrices $A$ and $B$ is denoted by is $A \oplus B.$

\section{Conformal partition of matrices and its operations}

In this section, we recall some relevant definitions and introduce the notion of sub-order conformity for partitioned matrices. We then naturally define the addition and product of these matrices, followed by the concepts of conformal product and $A$-conformal matrix.

\subsection{Conformal partition of matrices}

\begin{definition}[\cite{Horn}, 0.7.2] 
    If $\alpha_1, \dots, \alpha_t$ is a partition of $\{1, \dots, m\}$ and $\beta_1,\beta_2,\newline \ldots, \beta_s$ is a partition of $\{1, \dots, n\}$, then the submatrices $A[\alpha_i, \beta_j]$ form a partition of the matrix $A \in M_{m,n}(\mathbb{K}), 1 \leq i \leq t, 1 \leq j \leq s$. If $A \in M_{m,n}(\mathbb{K})$ and $B \in M_{n,p}(\mathbb{K})$ are partitioned so that the two partitions of $\{1, \dots, n\}$ coincide, the two matrix partitions are said to be conformal. 
\end{definition}

\begin{example}
Consider the matrix
\[
A=
\begin{bmatrix}
1 & 2 & 3 & 4\\
5 & 6 & 7 & 8\\
9 & 10 & 11 & 12
\end{bmatrix}
\in M_{3,4}(\mathbb{R}).
\]
Let
$
\alpha_1=\{1,2\},  \alpha_2=\{3\},
$ be a partition of $\{1,2,3\}$, and
$
\beta_1=\{1,2\}, \beta_2=\{3,4\},
$
be a partition of $\{1,2,3,4\}$.

Then the matrix $A$ is partitioned into the blocks
\[
A=
\left[
\begin{array}{cc|cc}
1&2&3&4\\
5&6&7&8\\
\hline
9&10&11&12
\end{array}
\right]
=
\begin{bmatrix}
A[\alpha_1,\beta_1] & A[\alpha_1,\beta_2]\\
A[\alpha_2,\beta_1] & A[\alpha_2,\beta_2]
\end{bmatrix},
\]
where
\[
A[\alpha_1,\beta_1]=
\begin{bmatrix}
1&2\\
5&6
\end{bmatrix},
\qquad
A[\alpha_1,\beta_2]=
\begin{bmatrix}
3&4\\
7&8
\end{bmatrix},
\]
\[
A[\alpha_2,\beta_1]=
\begin{bmatrix}
9&10
\end{bmatrix},
\qquad
A[\alpha_2,\beta_2]=
\begin{bmatrix}
11&12
\end{bmatrix}.
\]

Now consider the matrix
\[
B=
\begin{bmatrix}
1&0\\
2&1\\
3&0\\
4&1
\end{bmatrix}
\in M_{4,2}(\mathbb{R}).
\]

Partition the rows of $B$ by
$\beta_1=\{1,2\}, \beta_2=\{3,4\},$
and partition its columns by
$\gamma_1=\{1\},  \gamma_2=\{2\}.$ Then
\[
B=
\left[
\begin{array}{c|c}
1 & 0\\
2 & 1\\
\hline
3 & 0\\
4 & 1
\end{array}
\right]
=
\begin{bmatrix}
B[\beta_1,\gamma_1] & B[\beta_1,\gamma_2]\\
B[\beta_2,\gamma_1] & B[\beta_2,\gamma_2]
\end{bmatrix},
\]
where
\[
B[\beta_1,\gamma_1]
=
\begin{bmatrix}
1\\
2
\end{bmatrix},
\qquad
B[\beta_1,\gamma_2]
=
\begin{bmatrix}
0\\
1
\end{bmatrix},
\]
\[
B[\beta_2,\gamma_1]
=
\begin{bmatrix}
3\\
4
\end{bmatrix},
\qquad
B[\beta_2,\gamma_2]
=
\begin{bmatrix}
0\\
1
\end{bmatrix}.
\]

Since the partition of the columns of $A$ coincides with the partition of the rows of $B$, the matrix partitions are conformal. Consequently, the product $AB$ can be computed blockwise:
\[
(AB)_{ij}
=
\sum_{k}
A[\alpha_i,\beta_k]\,B[\beta_k,\gamma_j].
\] $\hfill{\diamond}$
\end{example}

To facilitate the extension of the theory to rectangular matrix products and block matrix structures, we reformulated the original definition and notation in a more operational way. The new definitions and notations emphasize the conditions for algebraic operations on rectangular block matrices and lay the groundwork for the study of rectangular block circulant matrices and related matrix products.

\begin{definition}[Sub-order of blocks partition]\label{suborder} \cite{Horn}{(0.7.1)}
    Let $A \in M_{m,n}(\mathbb{K})$ be a partitioned matrix according to the index sets $\alpha = \{\alpha_1, \dots, \alpha_t\}$ and $\beta = \{\beta_1, \dots, \beta_s\}$. We shall define the sub-order of row block $i$ as the cardinality $m_i = |\alpha_i|$, and the sub-order of column block $j$ as the cardinality $n_j = |\beta_j|$. 
\end{definition}
In what follows we use the notation $A=\left[ A[m_i,n_j]\right]$ for a block matrix, and we shall understand that $A[m_i,n_j]$ denotes a block of sub-order $m_i \times n_j$, in the sense of the Definition \ref{suborder}. Therefore, we introduce a new type of conformity between partitioned matrices, which we refer as \emph{sub-order conformity}. Following this definition we define the two corresponding operations, sum and product. 

\begin{definition}[Sub-order Conformity]\label{Sub-order Conformity} 
Let $A =\left[ A[m_i,p_j] \right] \in M_{m,n}(\mathbb{K})$ and $B = \left[B[q_i,n_j]\right] \in M_{n,p}(\mathbb{K})$ be partitioned matrices. 

We say that $A$ and $B$ have
\emph{sub-order conformity} if and only if the number of column blocks
of $A$ coincides with the number of row blocks of $B$, say $s=t$, and
\[
p_j=q_j,\qquad j=1,\ldots,s.
\]
In other words, the sub-order of each column block of $A$ coincides
with the sub-order of the corresponding row block of $B$.
\end{definition}

The notion of sub-order conformity allows us to define the addition and multiplication of partitioned matrices in a natural way. We note that, in \cite{Horn},  matrix partitions whose blocks have compatible dimensions for matrix multiplication are said to be conformal. In contrast, our notion of sub-order conformity is based solely on the compatibility of the sub-orders of the corresponding blocks.

Thus, for the sum $A+B$ to be well defined, the matrices must have the same sub-order conformity. In this case,
 $$(A+ B)[m_i, n_j] = A[m_i, n_j] + B[m_i, n_j].$$
 
The result is a new partitioned matrix where each block is the sum of the corresponding blocks of the original matrices. The conformal multiplication of two block matrices requires a specific compatibility between the partitions. For the product $AB$ to be defined, the partition of the columns of $A$ must agree with the partition of the rows of $B$. Under this compatibility condition, the $(i,j)$-th block of $AB$ is given by
$$
(AB)[m_i,n_j]=\sum_{k=1}^{s}A[m_i,p_k]\,B[p_k,n_j].
$$
Therefore, block matrix multiplication follows the same rule as ordinary matrix multiplication, with scalar entries replaced by compatible submatrices.

\subsection{Conformal product of  conformal partitioned matrices }
In this section we present the conformal product which provides a natural way of combining a partitioned matrix with an ordinary matrix while preserving the block structure. More precisely, the matrix $K$ in the definition below determines how the column blocks of A are linearly combined to form the blocks of a new partitioned matrix. 

\begin{definition}\label{Def conf prod}
Let $A=[A[m_i,n_j]]$ be a partitioned matrix with row blocks
\(m_1,\ldots,m_r\) and columns blocks
\(n_1=\cdots=n_s=n_0\).
Let $K=[k_{ij}]\in M_{s,p}(\mathbb K).$
The \textit{conformal product} of $A$ and $K$, denoted by
$A\circledc K$, is the partitioned matrix
\[
A\circledc K=[C_{ij}],
\]
whose $(i,j)$-th block is defined by
\[
C_{ij}
=
\sum_{t=1}^{s}
A[m_i,n_t]\, k_{tj},
\qquad
i=1,\ldots,r,\;
j=1,\ldots,p.
\]
The matrix $A\circledc K$ is called the
$A$-\textit{conformal matrix associated with} $K$
and it is denoted by $A_K = \left[A_K [m_i, n_j]\right]$.
\end{definition}


\begin{example}
Let $A=\begin{pmatrix}
\begin{bmatrix}
    a_{11}\\
    a_{21}
\end{bmatrix} &\begin{bmatrix}
    b_{11}\\
    b_{21}
\end{bmatrix}&\begin{bmatrix}
    c_{11}\\
    c_{21}
\end{bmatrix} \\[15pt]
 \begin{bmatrix}
     f_{11}
 \end{bmatrix}&  \begin{bmatrix}
     g_{11}
 \end{bmatrix}& \begin{bmatrix}
     h_{11}
 \end{bmatrix}
    \end{pmatrix} \in M_{3,3} (\mathbb{K})$ and consider the matrix 
    $ K=\begin{pmatrix}
        1&2&3\\
        -2&2&3\\
        1&1&1
    \end{pmatrix}.$ Then
\[
   A_K = A \circledc 
    K=
   \begin{pmatrix}
\begin{bmatrix}
    a_{11}-2b_{11}+c_{11}\\
    a_{21}-2b_{21}+c_{21}
\end{bmatrix} &\begin{bmatrix}
    2a_{11}+2b_{11}+c_{11}\\
    2a_{21}+2b_{21}+c_{21}
\end{bmatrix}&\begin{bmatrix}
    3a_{11}+3b_{11}+c_{11}\\
    3a_{21}+3b_{21}+c_{21}
\end{bmatrix}\\[15pt]
 \begin{bmatrix}
     f_{11}-2 g_{11}+h_{11}
 \end{bmatrix}&  \begin{bmatrix}
     2f_{11}+2 g_{11}+h_{11}
 \end{bmatrix}& \begin{bmatrix}
     3f_{11}+3 g_{11}+h_{11}
 \end{bmatrix}
    \end{pmatrix}.
\]
Note that the first column  of blocks in $A_K$ is 
$2 A[m_i, n_1]+ 2 A[m_i, n_2] + A[ m_i, n_3],$ with $n_1=n_2=n_3=1.$ The same reasoning can be checked for the others. 
\end{example}$\hfill{\diamond}$

\begin{example} 
   \[
   \begin{pmatrix}
       \begin{bmatrix}
           1&2&3\\
           1&0&1
       \end{bmatrix}&\begin{bmatrix}
           2&2&1\\
           1&1&1
       \end{bmatrix}\\[13pt]
       \begin{bmatrix}
           0&1&0
       \end{bmatrix}&\begin{bmatrix}
           1&1&0
       \end{bmatrix}
   \end{pmatrix}\circledc \begin{pmatrix}
       2&1\\
       3&1
   \end{pmatrix}= \begin{pmatrix}
       \begin{bmatrix}
           8&10&9\\
           5&3&5
       \end{bmatrix}&\begin{bmatrix}
           3&4&4\\
           2&1&2
       \end{bmatrix}\\[13pt]
       \begin{bmatrix}
           3&5&0
       \end{bmatrix}&\begin{bmatrix}
           1&2&0
       \end{bmatrix}
   \end{pmatrix}.
   \]
\end{example}$\hfill{\diamond}$

Let $A\in M_{m,n}(\mathbb{K})$, and let $CM_A(\mathbb{K})$ denote the set of $A$-conformal matrices
$$
CM_A(\mathbb{K})=\{A\circledc K:\,K\in M_{s,p}(\mathbb{K})\}.
$$

 \begin{remark}\label{operations A-conformal matrix}
With the following operations, the set $CM_A(\mathbb{K})$ forms a subring of $M_{m,p}(\mathbb{K})$. \ 
\begin{description}
        \item[a)] $(A\circledc K_1) \cdot (A\circledc K_{2}) = A\circledc (K_{1}K_{2})$ assuming that the product $K_1 K_2$ is defined.
        \item[b)] $(A\circledc K_1) + (A \circledc K_{2}) = A \circledc (K_1+K_2).$
    \end{description}
\end{remark}

\begin{proposition}
Let  $D\in M_{n}(\mathbb{K})$ and $K \in M_{s,p} (\mathbb{K})$ be invertible matrices, where $D$ is a diagonal matrix. Then
    \[
    (D\circledc K)^{-1}=K^{-1}\circledc D^{-1}
    \]
    provided that $D$ (or $K$) is a square matrix partitioned into $1\times 1$ blocks.
\end{proposition}
\begin{proof}
 First, we assume that $D=\left[D[1,1]\right]=\diag(d_{1},d_2,\ldots,d_n)$ is partitioned into square blocks of order $1$ and $K=\left[K[m_i,n_j]\right]=(k_{ij})$  with $m_i, n_j \neq 1$. Since $D$ is partitioned into $n$ of $1\times1$ blocks, the conformality of the product implies that $s=n$. Then

$$(D\circledc K)_{ij}=d_{i}k_{ij}
,\quad i, j=1, \dots, n,$$ 
and 
$$
   ( K^{-1}\circledc D^{-1})_{ij}=(K^{-1})_{ij}d_{j}^{-1}, i,j =1, \ldots, n  
$$

Now, for every $i,j=1,\ldots,n$,
\[
\begin{aligned}
((D\circledc K)(K^{-1}\circledc D^{-1}))_{ij}
&=
\sum_{t=1}^{n}
d_t\,k_{it}\,(K^{-1})_{tj}\,d_j^{-1} \\
&=
d_jd_j^{-1}
\sum_{t=1}^{n}
k_{it}(K^{-1})_{tj} \\
&=
\sum_{t=1}^{n}
k_{it}(K^{-1})_{tj} \\
&=
\delta_{ij},
\end{aligned}
\mbox{where } 
d_t\delta_{tj}=d_j\delta_{tj}.
\]
Hence,
$(D\circledc K)(K^{-1}\circledc D^{-1})=I_n.$ Similarly,
$(K^{-1}\circledc D^{-1})(D\circledc K)=I_n,$
and therefore
$(D\circledc K)^{-1}=K^{-1}\circledc D^{-1}.$

Now, let us assume that $D=[D[m_i,n_j]]=\diag(d_1,d_2\ldots,d_n)$ with $m_i, n_j \neq 1$, and $K=\left[K[1,1]\right]=(k_{ij})$ then it suffices to note that in the development of the product as we have $$d_t\,k_{it}\,(K^{-1})_{tj}\,d_j^{-1}=d_t\,d_j^{-1}\,k_{it}\,(K^{-1})_{tj},$$ 
the result follows.
\end{proof}
\section{Conformal permutation matrix}
Permutation matrices play a fundamental role in matrix theory and linear algebra. A permutation matrix is obtained from the identity matrix by permuting its rows or columns. Multiplication by a permutation matrix corresponds to a reordering of rows or columns: left multiplication permutes rows, whereas right multiplication permutes columns. Permutation matrices are orthogonal, so that $P^{-1}=P^{T}$. Moreover, they preserve many important matrix properties, including rank, determinant up to sign, and eigenvalues under similarity transformations. Because of these properties, permutation matrices are widely used in numerical linear algebra, graph theory in the study of structured matrices.
Let $\pi \in S_{n}$, where $S_{n}$ is the symmetric group.  Each element $\pi \in S_{n}$ corresponds to a permutation matrix $P_{\pi}= (p_{i,j})$ where $p_{i,j}=1$ if $j= \pi(i)$ and zero otherwise. A square matrix having in each row and column only one non-zero element is called a generalized permutation. We now introduce the definition of conformal permutation matrix.

\begin{definition}[conformal permutation matrix] \label{diagonal} Let \(\pi\in S_n\) and let \(P_\pi\) be a generalized permutation matrix.
Consider the block matrix
\[
D=\left[D[m_i,n_j]\right]_{i,j=1}^n,
\]
where
\[
D[m_i,n_j]=
\begin{cases}
d_i, & \text{if } j=i,\\
0_{m_i\times n_j}, & \text{if } j\neq i,
\end{cases}
\]
and each \(d_i\) is an \(m_i\times n_i\) matrix over \(\mathbb K\). The matrix
\[
D_{\pi}
=
D\circledc P_\pi
\]
is called the \emph{conformal permutation matrix} associated with
\(\pi\).

\end{definition}
According to Definition~\ref{Def conf prod}, $D\circledc P_\pi$ is denoted by $D_{P_\pi}$, but for simplicity and for better handling of operations we just denote it by $D_{\pi}$.\\

\begin{remark}
    From Remark~\ref{operations A-conformal matrix}, and considering $D$ as in the previous definition, it follows immediately: 
   \[
D_{\pi^{\ell}} = D\circledc P^\ell_\pi = D \circledc P_{\pi^{\ell}};\quad\textit{for all }\quad\ell\in\mathbb{N}.
    \]
Thus, $D_{\pi^{\ell}}$ is an $\ell$-conformal power of the generalized permutation matrix $P_{\pi}$, for $\ell\in\mathbb{N}.$
\end{remark}
\begin{example} Let $\pi_1= (1\, 2\, 3)$ and $P_{\pi_1}=\begin{pmatrix}
        0&1&0\\
        0&0&1\\
        1&0&0
    \end{pmatrix}$. Then,
    \small{
\[
D_{\pi_1}=
D\circledc P_{\pi_1}
=\]
\[
\begin{pmatrix}
        \begin{bmatrix}
            a_{11}&a_{12}&a_{13}\\
            a_{21}&a_{22}&a_{23}
        \end{bmatrix} & \begin{bmatrix}
            0\\ 
            0
        \end{bmatrix} &\begin{bmatrix}
            0&0\\
            0&0
        \end{bmatrix}\\[15pt]
        \begin{bmatrix}
            0&0&0
        \end{bmatrix} &\begin{bmatrix}
            b
        \end{bmatrix} &\begin{bmatrix}
            0 &0
        \end{bmatrix}\\[15pt]
        \begin{bmatrix}
            0&0&0\\
            0&0&0
        \end{bmatrix} &\begin{bmatrix}
            0\\
            0
        \end{bmatrix}&\begin{bmatrix}
            c_{11} &c_{12}\\
            c_{21} &c_{22}
        \end{bmatrix}
    \end{pmatrix}\circledc P_{\pi_1}=
    \begin{pmatrix}
        \begin{bmatrix}
            0&0\\
            0&0
        \end{bmatrix}&\begin{bmatrix}
            a_{11}&a_{12}&a_{13}\\
            a_{21}&a_{22}&a_{23}
        \end{bmatrix} & \begin{bmatrix}
            0\\ 
            0
        \end{bmatrix}\\[15pt]
     \begin{bmatrix}
            0 &0
        \end{bmatrix}&  \begin{bmatrix}
            0&0&0
        \end{bmatrix} &\begin{bmatrix}
            b
        \end{bmatrix} \\[15pt]
      \begin{bmatrix}
            c_{11} &c_{12}\\
            c_{21} &c_{22}
        \end{bmatrix}&   \begin{bmatrix}
            0&0&0\\
            0&0&0
        \end{bmatrix} &\begin{bmatrix}
            0\\
            0
\end{bmatrix}
\end{pmatrix},
    \]
    }
and, the $2$-conformal power of$P_{\pi_1}$ is
\small{
    \[D_{\pi ^{2}}=D\circledc P_{\pi_1^{2}}
=\\ \]
    \[
    \begin{pmatrix}
        \begin{bmatrix}
            a_{11}&a_{12}&a_{13}\\
            a_{21}&a_{22}&a_{23}
        \end{bmatrix} & \begin{bmatrix}
            0\\ 
            0
        \end{bmatrix} &\begin{bmatrix}
            0&0\\
            0&0
        \end{bmatrix}\\[15pt]
        \begin{bmatrix}
            0&0&0
        \end{bmatrix} &\begin{bmatrix}
            b
        \end{bmatrix} &\begin{bmatrix}
            0 &0
        \end{bmatrix}\\[15pt]
        \begin{bmatrix}
            0&0&0\\
            0&0&0
        \end{bmatrix} &\begin{bmatrix}
            0\\
            0
        \end{bmatrix}&\begin{bmatrix}
            c_{11} &c_{12}\\
            c_{21} &c_{22}
        \end{bmatrix}
    \end{pmatrix}\circledc P_{\pi_{1}^2}
    =
    \begin{pmatrix}
        \begin{bmatrix}
            0\\ 
            0
        \end{bmatrix}&\begin{bmatrix}
            0&0\\
            0&0
        \end{bmatrix}&\begin{bmatrix}
            a_{11}&a_{12}&a_{13}\\
            a_{21}&a_{22}&a_{23}
        \end{bmatrix} \\[15pt]
     \begin{bmatrix}
            b
        \end{bmatrix} &\begin{bmatrix}
            0 &0
        \end{bmatrix}&  \begin{bmatrix}
            0&0&0
        \end{bmatrix}\\[15pt]
      \begin{bmatrix}
            0\\
            0
        \end{bmatrix}&\begin{bmatrix}
            c_{11} &c_{12}\\
            c_{21} &c_{22}
        \end{bmatrix}&   \begin{bmatrix}
            0&0&0\\
            0&0&0
        \end{bmatrix}     
\end{pmatrix}.
    \]
    }
    $\hfill{\diamond}$
\end{example}
\vspace{0.7cm}

\medskip
We will now focus specifically on conformal permutation matrices in the following form:

For the special case when, 
\begin{equation} \label{pis}\pi_s(k)=k+s\,  (\mbox{\rm{mod}}\,n), \end{equation}

 we denote the conformal permutation matrix $D\circledc P_{\pi_s}$ by $H_{\pi_s}$, this is
\begin{equation}\label{ciclic conformal definition}
    H_{\pi_s}=
    \begin{cases}
        h_i, & \text{if } j=\pi_s(i),\\
        0, & \text{if } j\neq\pi_s(i),
    \end{cases}
\end{equation}
where each \(h_i\) is a nonzero \(m_i\times m_{\pi_s(i)}\) matrix over
\(\mathbb{K}\).
\medskip
\begin{remark}\label{prod not conformal}
 Note that the $\ell-$th power of the matrix $H_{\pi_s}$ is $(H_{\pi_s})^{\ell}=\underbrace{H_{\pi_s}\cdots H_{\pi_s}}_\text{$\ell$-times} $. However, this is different from calculating the $\ell-$conformal power of $H_{\pi_s}$, which is $D\circledc P^\ell_{\pi_s}.$
\end{remark}
\medskip

\begin{example}\label{example conformal perm 1} Let $\pi_1= (1\, 2\, 3)$ and $D=\diag(h_1, h_2, h_3),$ with $h_1=\begin{bmatrix}
            a_{11}&a_{12}\\
            a_{21}&a_{22}
        \end{bmatrix},$ $h_2=\begin{bmatrix}
            b_{11}&b_{12}&b_{13}\\
            b_{21}&b_{22}&b_{23}
        \end{bmatrix}, h_3=\begin{bmatrix}
            c_{11}&c_{12}\\
            c_{21}&c_{22}\\
            c_{31}&c_{32}
        \end{bmatrix}.$ Then, 
    \[
    H_{\pi_1}=D \circledc P_{\pi_1} =\begin{pmatrix}
        \begin{bmatrix}
            0&0\\
            0&0
        \end{bmatrix}&\begin{bmatrix}
            a_{11}&a_{12}\\
            a_{21}&a_{22}
        \end{bmatrix} & \begin{bmatrix}
            0&0&0\\ 
            0&0&0
        \end{bmatrix}\\[15pt]
     \begin{bmatrix}
            0&0 \\
            0&0 
        \end{bmatrix}&  \begin{bmatrix}
            0&0\\
            0&0
        \end{bmatrix} &\begin{bmatrix}
            b_{11}&b_{12}&b_{13}\\
            b_{21}&b_{22}&b_{23}
        \end{bmatrix} \\[15pt]
      \begin{bmatrix}
            c_{11}&c_{12}\\
            c_{21}&c_{22}\\
            c_{31}&c_{32}
        \end{bmatrix}&   \begin{bmatrix}
            0&0\\
            0&0\\
            0&0
        \end{bmatrix} &\begin{bmatrix}
            0&0&0\\
            0&0&0\\
            0&0&0
        \end{bmatrix}
    \end{pmatrix}:=\begin{pmatrix}
        0  & h_1 &0\\
        0  & 0   &h_2\\
        h_3& 0   &0
    \end{pmatrix}
    \]
    \medskip

\[
(H_{\pi_1})^2=\begin{pmatrix}
    0& 0& h_1h_2\\
    h_2h_3&0&0\\
    0& h_3h_1&0
\end{pmatrix}.
\]

    \[
     (H_{\pi_1})^3=\begin{pmatrix}
        h_1h_2h_3& 0       &0\\
           0     &h_2h_3h_1&0\\
           0     &  0      &h_3h_1h_2
    \end{pmatrix}.
    \]
    \medskip
 
The product is possible because the matrices $h_1$, $h_2$ and $h_3$, are of orders $m_1\times m_2$, $m_2\times m_3$ and $m_3\times m_1$, respectively. 
We note that the $2-$conformal power of $H_{\pi_1}$ is 
    \[
    H_{\pi^2_1}= \begin{pmatrix}
     h_1&0&0\\
     0&h_2&0\\
     0&0&h_3
    \end{pmatrix}\circledc 
    \begin{pmatrix}
        0&0&1\\
        1&0&0\\
        0&1&0
    \end{pmatrix}=\begin{pmatrix}
        0&0&h_1\\
        h_2&0&0\\
        0&h_3&0
    \end{pmatrix}
    \]
    Thus, $(H_{\pi_{1}})^\ell\neq H_{\pi_{1}^\ell}$. $\hfill{\diamond}$
\end{example}
\medskip

\begin{remark} It is worth noticing that when $\pi_s$ is as in \eqref{pis}, 
\[
(H_{\pi_s})_{i,j}\neq 0
\quad \Longleftrightarrow \quad
j\equiv i+s \pmod{n},
\]
where the nonzero entry is $h_i$. Then, for every $\ell \geq 1$, the nonzero entries of $(H_{\pi_s} )^{\ell}$ are
products of the form
\[
h_i h_{i+s}h_{i+2s}\cdots h_{i+(\ell-1)s},
\]
where the indices are taken modulo $n$. Let 
\begin{equation}
    W_{\pi_s}=\diag(h_1 h_{1+s}h_{1+2s}\cdots h_{1+(\ell-1)s},\ldots\ldots,h_n h_{n+s}h_{n+2s}\cdots h_{n+(\ell-1)s})\circledc P_{\pi_s}
\end{equation}
Thus, the following relationship between the $\ell-$th power of $H_{\pi_s}$ and the $\ell-$conformal power of $W_{\pi_s}$ follows:
\[
(H_{\pi_s})^\ell=W_{\pi_s^\ell}.
\]
\end{remark}

Throughout the text we will consider consider \begin{equation} \label{orderd} d=\mathcal{O}(\pi_s)=\frac{n}{\gcd (s,n)}.\end{equation}

\begin{lemma}\label{Def Dj matrices}
Let  $i=1,\ldots,n$, $\pi_s$ as in \eqref{pis} and $d$ as in \eqref{orderd}.
 Then $(H_{\pi_s})^{d}=\diag(D_1,\ldots,D_n)$, where 
\begin{equation} 
D_{i}=\prod_{\ell=0}^{d-1}h_{\pi_s^{\ell}(i)}\in M_{m_i,m_{i}}(\mathbb{K}).
\end{equation}
\end{lemma}
\begin{proof}
 It is clear that each $D_i, i=1, \ldots, n$, is a matrix of order $m_i\times m_i.$ From the relation described above, 
 \begin{eqnarray*}
 (H_{\pi_s})^d&=&W_{\pi_s^d}\\
             &=&\diag(h_1 h_{1+s}h_{1+2s}\cdots h_{1+(d-1)s},\ldots\ldots,h_n h_{n+s}h_{n+2s}\cdots h_{n+(d-1)s})\circledc P^d_{\pi_s}\\
             &=&\diag(D_1,\ldots,D_n)\circledc P^d_{\pi_s}
 \end{eqnarray*}
 and since $P^d_{\pi_s}=I_n$ the result follows.
\end{proof}
\medskip

\begin{lemma}\label{realtion D_j} 
Let $\pi_s$ as in \eqref{pis} with order $d$ where $\gcd(s,n)=1$,  and the matrix $\diag(D_1,\ldots,D_n)$ where each $D_i \in M_{m_i, m_i}(\mathbb{K})$ is as in Lemma \ref{Def Dj matrices}. Consider $H_i=\prod\limits_{k=0}^{n-2}h_{\pi^k_s(i)}$. Then,
        \[
         H_iD_{i+n_0}=D_i H_i, \quad i=1,\ldots, n \pmod{n}
        \]
        where $n_0=n-s.$
\end{lemma}

\begin{proof} If $\gcd(s,n)=1$ we have $d=n$ and then
\begin{align*}
   D_{i+n_0}&=h_{i+n_0}h_{\pi_s(i+n_0)}h_{\pi^2_s(i+n_0)}\cdots h_{\pi^{n-1}_s(i+n_0)}\\
   &=h_{i+n_0}h_{i+n_0+s}h_{i+n_0+2s}\cdots h_{i+n_0+(n-1)s}\\
   &=h_{i+(n-1)s}h_{i}h_{i+s}\cdots h_{i+(n-2)s} 
\end{align*}
Thus, 
\begin{align*}
    H_iD_{i+n_0}&=(h_ih_{\pi_s(i)}h_{\pi^2_s(i)}\cdots,h_{\pi^{n-2}_s(i)})( h_{i+n_0}h_{\pi_s(i+n_0)}h_{\pi^2_s(i+n_0)}\cdots h_{\pi^{n-1}_s(i+n_0)})\\
    &=(h_ih_{i+s}h_{i+2s}\cdots h_{i+(n-2)s} )(h_{i+(n-1)s}h_{i}h_{i+s}\cdots h_{i+(n-2)s})\\
    &=(h_ih_{i+s}h_{i+2s}\cdots h_{i+(n-2)s}h_{i+(n-1)s})(h_{i}h_{i+s}\cdots h_{i+(n-2)s})\\
    &=D_i H_i
\end{align*}
\end{proof}
\medskip
\begin{example}
    In Example~\ref{example conformal perm 1}, we have $D_1=h_1h_3h_2$; $D_2=h_2h_1h_3$, $D_3=h_3h_2h_1$ and $H_1=h_1h_3$; $H_2=h_2h_1$; $H_3=h_3h_2$. Therefore,
    \begin{align*}
        H_1D_2&=(h_1h_3)(h_2h_1h_3)=(h_1h_3h_2)(h_1h_3)=D_1 H_1\\
        H_2D_3&=(h_2h_1)(h_3h_2h_1)=(h_2h_1h_3)(h_2h_1)=D_2 H_2\\
        H_3D_1&=(h_3h_2)(h_1h_3h_2)=(h_3h_2h_1)(h_3h_2)=D_3 H_3.
    \end{align*}$\hfill{\diamond}$
\end{example}

\begin{remark}
    If $s|n$  and $\gcd(n,s)\neq 1$, then $n = \ell s$ for some positive $\ell$. 
    We then assume that $h_1,h_2,\ldots,h_{n}$ have appropriate order and we can construct the $\tilde{m}_j\times \tilde{m}_{j+1}$ matrices $$\tilde{h}_j = \diag(h_{1+(j-1)s}, \ldots, h_{s+(j-1)s}), j=1,2,\ldots,\ell.$$ 
    Thus, we define
    $$
\tilde{H}_{\pi_1}=\diag(\tilde{h}_1,\ldots,\tilde{h}_\ell)\circledc P_{\pi_1}$$ 
which is a conformal permutation matrix associated to $\pi_1,$ that can be reduced to the conditions of the previous Lemma~\ref{realtion D_j}
 \end{remark}

\begin{example}
    Let $n=4$, $s=2$, then $\gcd(s,n)=2$ and $d=2$. Consider $h_1, h_2, h_3, h_4$ matrices of orders $m_1 \times m_3, m_2 \times m_4, m_3 \times m_1, m_4 \times m_2, $, respectively. In this case, we have
    \[
    H_{\pi_2}=\begin{pmatrix}
        0&0&h_1&0\\
        0&0&0&h_2\\
        h_3&0&0&0\\
        0&h_4&0&0
    \end{pmatrix}
    \]
\[
\tilde{h}_1=\begin{pmatrix}
    h_1&0\\
    0&h_2
\end{pmatrix};\quad
\tilde{h}_2=\begin{pmatrix}
    h_3&0\\
    0&h_4
\end{pmatrix}
\]
\[
\tilde{H}_{\pi_1}=
\begin{pmatrix}
   0&\tilde{h}_1\\
   \tilde{h}_2&0
\end{pmatrix}.
\]
Thus, the study of $H_{\pi_2}$ can be reduced to $\tilde{H}_{\pi_1}$. $\hfill{\diamond}$
\end{example}
\medskip

\begin{theorem}\label{Eigenpars preserved} Consider the set of matrices 
$\{D_1, D_2, \dots, D_n\}$ as in Lemma~\ref{realtion D_j} and $\gcd(s,n)=1.$ Then the matrices $D_j$ will share an identical subset of eigenvalues of size at most $m_{n_0}$. Specifically, if $\lambda$ is an eigenvalue of $D_{n_0}$, it will also be an eigenvalue of any $D_j$, for all $j\neq n_0$.
 \end{theorem}
\begin{proof}
Let $n_0 = n-s$, and let $\lambda$ be an eigenvalue of $D_{n_0}$ with
associated eigenvector $u_{n_0}\neq0$, that is, $$D_{n_0}u_{n_0}=\lambda_{n_0}u_{n_0}.$$ 
By Lemma ~\ref{realtion D_j},  $H_nD_{n+n_0}=D_nH_n$ thus,
\[
D_n(H_n u_{n_0})=H_n D_{n+n_0} u_{n_0}=H_n D_{n_0}u_{n_0}=\lambda_{n_0}(H_n u_{n_0}).
\]
Therefore, $u_n=H_n u_{n_0}$ is an eigenvector of $D_n$ (provided it is non-zero) associated to the eigenvalue $\lambda_{n_0}.$ Furthermore, if $u_{j+n_0}\neq 0$ is such that $D_{j+n_0}u_{j+n_0}=\lambda_{n_0}u_{j+n_0},$ then
\[
D_j(H_ju_{j+n_0})=H_j D_{j+n_0}u_{j+n_0}=H_j(\lambda_{n_0}u_{j+n_0})=\lambda_{n_0}(H_ju_{j+n_0}).
\]
Thus, $u_j=H_ju_{j+n_0}$, $j\neq n_0$ is an eigenvector of the matrices $D_j$ associated to the eigenvalue $\lambda_{n_0}.$ 
\end{proof}
\begin{example}
    In Example~\ref{example conformal perm 1}, $n=3$, $s=1$ and $n_0=2.$ Furthermore, the matrices $D_j, j=1,2,3$ in Lemma \ref{Def Dj matrices} are:
    \begin{align*}
        D_1&=h_1h_2h_3;& D_2&=h_2h_3h_1; &D_3=h_3h_1 h_2,
    \end{align*}
    \begin{align*}
        H_1&=h_1h_2;& H_2&=h_2h_3; &H_3=h_3h_1,
    \end{align*}
which are $2\times2$, $2\times 3$ and $3\times 2$ block matrices, respectively. 

If $(\lambda_i^{(2)},v_i^{(2)}), i=1,2,$ are eigenpairs of $D_2$ then, by Theorem~\ref{Eigenpars preserved}, $(\lambda_i^{(2)}, H_3 v_i^{(2)}),$ $i= 1,2$, are eigenpairs of $D_3$ and, $(\lambda_i^{(2)},H_1H_3v_i^{(2)}), i=1,2,$ are eigenpairs of $D_1,$    provided the vectors $H_3 v_i^{(2)}$, $H_1H_3v_i^{(2)}$ are nonzero.
  \medskip
  Therefore, if $(\lambda^{(3)}, v^{(3)})$ is the third eigenpair of $D_3$, the matrix 
  $$
  X=\small{\begin{pmatrix}
      [H_1H_3v_1^{(2)}]^T&[H_1H_3v_2^{(2)})]^T&0&0&0&0&0\\
      0&0&[v_1^{(2)}]^T&[v_2^{(2)}]^T&0&0&0\\
      0&0&0&0&[H_3v_1^{(2)}]^T&[H_3v_2^{(2)}]^T&[v^{(3)}]^T
  \end{pmatrix}}
  $$
 satisfy,
 \[
 (H_{\pi_1})^3=X\diag\left(\lambda_1^{(2)},\lambda_2^{(2)},\lambda_1^{(2)},\lambda_2^{(2)},\lambda_1^{(2)},\lambda_2^{(2)},\lambda^{(3)}\right)X^T.
 \]$\hfill{\diamond}$
\end{example}

\begin{remark} Assuming $m_1 \leq m_2 \leq \cdots \leq m_n$, the smallest block, $D_{n_0}$ in Theorem \ref{Eigenpars preserved}, serves as the base of the recursive isospectral reduction process associated with $(H_{\pi_s})^d$. This process transforms the original eigenvalue problem into an equivalent sequence of lower-dimensional problems while preserving the relevant spectral information. As a result, it provides a useful framework for investigating spectral stability, singularity, and controllability properties of $(H_{\pi_s})^d$ \cite{Buni2014,Reis}.\end{remark}

\subsection{Diagonalization of the conformal permutation matrix $H_{\pi_s}$}
In what follows we diagonalize the matrix $H_{\pi_s}$. Before doing so, we introduce some useful notation. 

\noindent Let $\left[ I[m_i, m_j]\right]$ be the matrix such that:
\[
I[m_i,m_j]
=
\begin{cases}
I_{m_i}, & \text{if } i=j,\\[2mm]
0_{m_i\times m_j}, & \text{if } i\neq j,
\end{cases}
\] 
Consider $D^{(\omega)}= \diag(\omega^0,\omega^1,\ldots,\omega^{n-1})$, $\omega^n=1$ and $X_n= \left[ I[m_i, m_i]\right]\circledc  D^{(\omega)}.$

\begin{lemma} Let $1\le s\le n-1$ and $\pi_s(k)=k+s\,  (\rm{mod}\,n)$ and $D^{(\omega)}$ as written above. Consider $H_{\pi_s}$ as in Equation \ref{ciclic conformal definition}.\\

Then, 
    \[
    X^\ell_n H_{\pi_s}X_n^{-\ell}=\omega^{-s\ell} H_{\pi_s},\quad \ell=1,\ldots,n-1.
    \]
\end{lemma}

\begin{proof}
Let $X_n =\omega^0I_{m_1} \oplus  \omega^{1}I_{m_2} \oplus \cdots \omega^{n-1}I_{m_s}$,
and  \[
(H_{\pi_s})_{ij}\neq 0 \quad \textit{when} \quad j\equiv i+s \pmod n.
\]
Then, since $X_n$ is a block-diagonal invertible matrix, the block diagonal entries of $X_n$ and its inverse are given, respectively, by
\[
(X_n)_{ii}=\omega^{i-1}I_{m_i}, \qquad
(X_n^{-1})_{ii}=\omega^{-(i-1)}I_{m_i},
\]
for $i=1,\ldots,n$. Therefore, for every pair $(i,j)$,
\[
(X_n H_{\pi_s} X_n^{-1})_{ij}
=
(X_n)_{ii}\,h_{ij}\,(X_n^{-1})_{jj}
=
\omega^{i-1}h_{ij}\omega^{-(j-1)}.
\]

If $h_{ij}\neq 0$, then by hypothesis $ j\equiv i+s \pmod n.$
Hence, $j-1\equiv i+s-1 \pmod n,$ and consequently,
\[
\omega^{i-1}\omega^{-(j-1)}
=
\omega^{\,i-1-(j-1)}
=
\omega^{-s}.
\]

Thus, every nonzero entry of $X_n H_{\pi_s}X_n^{-1}$ is multiplied by the same factor $\omega^{-s}$, namely
\[
(X_n H_{\pi_s}X_n^{-1})_{ij}
=
\omega^{-s}h_{ij}.
\]

As the null entries remain zero, it follows that
\[
X_n H_{\pi_s}X_n^{-1}
=
\omega^{-s}H_{\pi_s},
\]
and the result follow.
\end{proof}

\begin{theorem}
If $(\lambda,v)$ is an eigenpair of $H_{\pi_s}$, then the eigenpairs of the matrix $H_{\pi_s}$ are $(\lambda\omega^0,v),(\lambda\omega^1,X_nv), \ldots,(\lambda\omega^{n-1},X^\ell_nv),$ $\ell=1,\ldots,n-1.$ 
\end{theorem}
\begin{proof}
    Note that, recalling the definition of $X_n$, we have $X^{\ell}_nv\neq 0$ whenever $v\neq0.$ Thus, the result is an immediate consequence of the lemma above.
\end{proof}
The next theorem shows that the matrix $H_{\pi_1}$ is block diagonalizable via a conformal product of two matrices that involve the Discrete Fourier matrix and a diagonal block matrix obeying certain conditions.

\begin{theorem}\label{diagonalization H} Let $\{h_k\}$ be a family of $m_k\times m_{k+1}$ rectangular matrices with $m_{n+1}=m_1$, and $\gamma$ be an invertible and diagonalizable matrix satisfying
\begin{equation}\label{gamma_n}
\gamma^n=h_1h_2\cdots h_n.
\end{equation}
Let $\Lambda_{\gamma}=\diag(\gamma,\omega\gamma,\ldots,\omega^{n-1}\gamma)$, where $\omega$ is a primitive $n$-th root of unity and for $i=1,\ldots,n$,

\begin{equation}\label{gamma_i} \gamma_i=h_{n-i+1}\gamma_{i-1}\gamma^{-1},\end{equation}
 with $\gamma_0=I_{m_1}$. Then, when $m_1\cdot n=\sum\limits_{i=1}^n m_i$
\begin{equation}
    H_{\pi_1}=F_{\pi_1}\Lambda_{\gamma}F^{-1}_{\pi_1}.
\end{equation}
Where, $F_{\pi_1}=\diag(\gamma_n,\gamma_{n-1},\ldots,\gamma_1)\circledc F$, with $F$ the discrete Fourier transform. Thus, $H_{\pi_1}$ is a block-diagonalizable matrix.
\end{theorem}
\begin{proof}
\small{
    \begin{align*}
       H_{\pi_1}\begin{pmatrix}
                   \gamma_n\\
                   \gamma_{n-1}\omega^i\\
                   \gamma_{n-2}\omega^{2i}\\
                   \vdots\\
                   \gamma_{1}\omega^{(n-1)i}\\
                \end{pmatrix} &= \begin{pmatrix}
                    0&h_1&0&0&\cdots&0\\
                    0&0&h_2&0&\cdots&0\\
                    \vdots&\vdots&&\ddots&&0\\
                    0&0&\cdots&&\ddots&h_{n-1}\\
                    h_n&0&\cdots&\cdots&&0
                \end{pmatrix} 
                \begin{pmatrix}
                   I_{m_1}\\
                   h_2h_3\cdots h_{n-1}h_n\gamma^{-(n-1)}\omega^i\\
                   h_3h_4\cdots h_{n-1}h_n\gamma^{-(n-2)}\omega^{2i}\\
                   \vdots\\
                   h_n\gamma^{-1}\omega^{(n-1)i}\\
                \end{pmatrix}\\[6pt]
                &= \begin{pmatrix}
                   h_1h_2\cdots h_n\gamma^{-(n-1)}\omega^i\\
                   h_2h_3\cdots h_{n-1}h_n\gamma^{-(n-2)}\omega^{2i}\\
                  h_3h_4\cdots h_{n-1}h_n\gamma^{-(n-3)}\omega^{3i}\\
                   \vdots\\
                   h_n\\
                \end{pmatrix}\\[6pt]
              &=\begin{pmatrix}
                   \gamma^n\gamma^{-(n-1)}\omega^i\\
                   h_2h_3\cdots h_{n-1}h_n\gamma^{-(n-2)}\gamma^{-1}\gamma\omega^{2i}\\
                  h_3h_4\cdots h_{n-1}h_n\gamma^{-(n-3)}\gamma^{-1}\gamma\omega^{3i}\\
                   \vdots\\
                   h_n\gamma^{-1}\gamma\omega^{i}\omega^{(n-1)i}\\
                \end{pmatrix} \\[1pt]
                &=\begin{pmatrix}
                   I_{m_1}\\
                   h_2h_3\cdots h_{n-1}h_n\gamma^{-(n-1)}\omega^{i}\\
                  h_3h_4\cdots h_{n-1}h_n\gamma^{-(n-2)}\omega^{2i}\\
                   \vdots\\
                   h_n\gamma^{-1}\omega^{(n-1)i}\\
                \end{pmatrix}\circledc\gamma\omega^k\\[6pt]
                &=\begin{pmatrix}
                   \gamma_n\\
                   \gamma_{n-1}\omega^i\\
                   \gamma_{n-2}\omega^{2i}\\
                   \vdots\\
                   \gamma_{1}\omega^{(n-1)i}\\
                \end{pmatrix}\circledc\gamma\omega^i .\end{align*}}
\end{proof}

\begin{remark}\label{parcialmente diagonalizable}
    The following example shows that if $m_1\cdot n\neq\sum\limits_{i=1}^n m_i$ then by the previous theorem we will obtain exactly  $m_1\cdot n$ formulas for eigenvalues and eigenvectors of $H_{\pi_1}.$ 
\end{remark}

\begin{example}
We consider the following matrices,

\[
h_1=\begin{pmatrix}
    3&0&0\\
    0&5&0
\end{pmatrix}; \quad h_2=\begin{pmatrix}
    1&0&0\\
    0&1&0\\
    0&0&0
\end{pmatrix}; \quad h_3=\begin{pmatrix}
    1&0\\
    0&1\\
    0&0
\end{pmatrix}.
\]
Then,
\[
H_{\pi_1}=\begin{pmatrix}
    0&0&3&0&0&0&0&0\\
    0&0&0&5&0&0&0&0\\
    0&0&0&0&0&1&0&0\\
    0&0&0&0&0&0&1&0\\
    0&0&0&0&0&0&0&0\\
    1&0&0&0&0&0&0&0\\
    0&1&0&0&0&0&0&0\\
    0&0&0&0&0&0&0&0
\end{pmatrix}.
\]
In this case,
\[
\gamma_1=\begin{pmatrix}
    3^{-\frac{1}{3}}&0\\
    0&5^{-\frac{1}{3}}\\
    0&0
\end{pmatrix};\,\,\gamma_2=\begin{pmatrix}
    3^{-\frac{2}{3}}&0\\
    0&5^{-\frac{2}{3}}\\
    0&0
\end{pmatrix};\,\,\gamma_3=\begin{pmatrix}
    1&0\\
    0&1
\end{pmatrix};\,\, \gamma^{-1}=\begin{pmatrix}
    3^{-\frac{1}{3}}&0\\
    0&5^{-\frac{1}{3}}
\end{pmatrix},
\]
and
\[
F_{\pi_1}=\begin{pmatrix}
    1&0&1&0&1&0\\
    0&1&0&1&0&1\\
    3^{-\frac{2}{3}}&0&3^{-\frac{2}{3}}\omega&0&3^{-\frac{2}{3}}\omega^2&0\\
    0&5^{-\frac{2}{3}}&0&5^{-\frac{2}{3}}\omega&0&5^{-\frac{2}{3}}\omega^2\\
    0&0&0&0&0&0\\
    3^{-\frac{2}{3}}&0&3^{-\frac{2}{3}}\omega^2&0&3^{-\frac{2}{3}}\omega&0\\
    0&5^{-\frac{2}{3}}&0&5^{-\frac{2}{3}}\omega^2&0&5^{-\frac{2}{3}}\omega\\
    0&0&0&0&0&0\\
\end{pmatrix}.
\]

Computing $H_{\pi_1}\begin{pmatrix}
    \gamma_3\\
    \gamma_2\\
    \gamma_1
\end{pmatrix}$, we verify that  $H_{\pi_1}\begin{pmatrix}
    \gamma_3\\
    \gamma_2\\
    \gamma_1
\end{pmatrix}=
\begin{pmatrix}
    \gamma_3\\
    \gamma_2\\
    \gamma_1
\end{pmatrix}\circledc \gamma$, more specifically, we have 
\begin{align*}
H_{\pi_1}\begin{pmatrix}
    \gamma_3\\
    \gamma_2\\
    \gamma_1
\end{pmatrix}&=
\begin{pmatrix}
    0&0&3&0&0&0&0&0\\
    0&0&0&5&0&0&0&0\\
    0&0&0&0&0&1&0&0\\
    0&0&0&0&0&0&1&0\\
    0&0&0&0&0&0&0&0\\
    1&0&0&0&0&0&0&0\\
    0&1&0&0&0&0&0&0\\
    0&0&0&0&0&0&0&0
\end{pmatrix}
\begin{pmatrix}
    1&0\\
    0&1\\
    3^{-\frac{2}{3}}&0\\
    0&5^{-\frac{2}{3}}\\
    0&0\\
    3^{-\frac{2}{3}}&0\\
    0&5^{-\frac{2}{3}}\\
    0&0\\
\end{pmatrix}\\
&=
\begin{pmatrix}
    1&0\\
    0&1\\
    3^{-\frac{2}{3}}&0\\
    0&5^{-\frac{2}{3}}\\
    0&0\\
    3^{-\frac{2}{3}}&0\\
    0&5^{-\frac{2}{3}}\\
    0&0\\
\end{pmatrix}\circledc 
\begin{pmatrix}
    3^{\frac{1}{3}}&0\\
    0&5^{\frac{1}{3}}
\end{pmatrix}.
\end{align*}

From this, we can conclude that
\[
\left( 3^{\frac{1}{3}}, (1,0,3^{-\frac{2}{3}},0,0,3^{-\frac{2}{3}},0,0)^T \right)\]
and
\[\left( 5^{\frac{1}{3}}, (0,1,0,5^{-\frac{2}{3}},0,0,5^{-\frac{2}{3}},0)^T \right)\]
are two eigenpairs of $H_{\pi_1}.$
Similarly, starting from 
\[H_{\pi_1}\begin{pmatrix}
    \gamma_3\\
    \gamma_2\omega\\
    \gamma_1\omega^2
\end{pmatrix}=
\begin{pmatrix}
    \gamma_3\\
    \gamma_2\omega\\
    \gamma_1\omega^2
\end{pmatrix}\circledc \gamma\omega\textit{ and }
H_{\pi_1}\begin{pmatrix}
    \gamma_3\\
    \gamma_2\omega^2\\
    \gamma_1\omega
\end{pmatrix}=
\begin{pmatrix}
    \gamma_3\\
    \gamma_2\omega^2\\
    \gamma_1\omega
\end{pmatrix}\circledc \gamma\omega^2,
\]
we can compute four eigenpairs of $H_{\pi_1}.$ Thus, we can compute six eigenvalues and eigenvectors of $H_{\pi_1}.$
\end{example} $\hfill{\diamond}$

\begin{coro} Let $\Lambda_{\gamma}=\diag(\gamma,\omega\gamma,\ldots,\omega^{n-1}\gamma)$.
   If $s|n$, $\gamma^n$, $\{h_i\}$ are as in Theorem \ref{diagonalization H}, and \begin{equation}
    H_{\pi_s}=F_{\pi_s}\Lambda_{\gamma} \otimes I_sF^{-1}_{\pi_s},
\end{equation}
then, $F_{\pi_s}=\diag(\gamma_n,\gamma_{n-1},\ldots,\gamma_1)F\otimes I_s$. Thus, $H_{\pi_s}$ is a block-diagonalizable matrix. 
\end{coro}

\begin{proof}
    It is a direct consequence of Theorem~\ref{diagonalization H}.
\end{proof}
We will designate $F_{\pi_s}$ by Conformal Discrete Fourier Transform associated to $H_{\pi_s}$ and we present an example of such matrix. 

\begin{example}
    Consider $n=4$, then the Conformal Discrete Fourier Transform of order $m=m_1+m_2+m_3+m_4$ is:
    \[F_{\pi_1}=\begin{pmatrix}
        \gamma_4 & \gamma_4 &\gamma_4&\gamma_4\\
        \gamma_3&-i\gamma_3&-\gamma_3&i\gamma_3\\
        \gamma_2&-\gamma_2&\gamma_2&-\gamma_2\\
        \gamma_1&i\gamma_1&-\gamma_1&-i\gamma_1
    \end{pmatrix},
    \]
where $\gamma_1=h_4\gamma^{-1}$; $\gamma_2=h_3h_4\gamma^{-2}$; $\gamma_3=h_2h_3h_4\gamma^{-3}$; $\gamma_4=I_{m_1}$ and $\gamma^4=h_1h_2h_3h_4$. If $m_1=2$; $m_2=3$; $m_3=4$; $m_4=3$ and
$$h_1 = \begin{pmatrix} 1 & 0 & 0 \\ 0 & 1 & 0 \end{pmatrix};\,
    h_2 = \begin{pmatrix} 2 & 0 & 0 & 0 \\ 0 & 3 & 0 & 0 \\ 0 & 0 & 0 & 0 \end{pmatrix};\,    
    h_3 = \begin{pmatrix} 1 & 0 & 0 \\ 0 & 1 & 0 \\ 0 & 0 & 0 \\ 0 & 0 & 0 \end{pmatrix};\,   
    h_4 = \begin{pmatrix} 1 & 0 \\ 0 & 1 \\ 0 & 0 \end{pmatrix}
    $$
    then
    $$
\gamma_1 =\begin{pmatrix} 2^{-1/4} & 0 \\ 0 & 3^{-1/4} \\ 0 & 0 \end{pmatrix};\,
\gamma_2 = \begin{pmatrix} 2^{-1/2} & 0 \\ 0 & 3^{-1/2} \\ 0 & 0 \\ 0 & 0 \end{pmatrix};\,
\gamma_3 = \begin{pmatrix} 2^{1/4} & 0 \\ 0 & 3^{1/4} \\ 0 & 0 \end{pmatrix};\,
\gamma_4 = \begin{pmatrix} 1 & 0 \\ 0 & 1 \end{pmatrix}.
$$
Therefore,
\[
F_{\pi_1}=\left(\begin{array}{cc:cc:cc:cc}
1 & 0 & 1 & 0 & 1 & 0 & 1 & 0 \\
0 & 1 & 0 & 1 & 0 & 1 & 0 & 1 \\ \hdashline
2^{1/4} & 0 & -i 2^{1/4} & 0 & -2^{1/4} & 0 & i 2^{1/4} & 0 \\
0 & 3^{1/4} & 0 & -i 3^{1/4} & 0 & -3^{1/4} & 0 & i 3^{1/4} \\
0 & 0 & 0 & 0 & 0 & 0 & 0 & 0 \\ \hdashline
2^{-1/2} & 0 & -2^{-1/2} & 0 & 2^{-1/2} & 0 & -2^{-1/2} & 0 \\
0 & 3^{-1/2} & 0 & -3^{-1/2} & 0 & 3^{-1/2} & 0 & -3^{-1/2} \\
0 & 0 & 0 & 0 & 0 & 0 & 0 & 0 \\
0 & 0 & 0 & 0 & 0 & 0 & 0 & 0 \\ \hdashline
2^{-1/4} & 0 & i 2^{-1/4} & 0 & -2^{-1/4} & 0 & -i & 0 \\
0 & 3^{-1/4} & 0 & i 3^{-1/4} & 0 & -3^{-1/4} & 0 & -i \\
0 & 0 & 0 & 0 & 0 & 0 & 0 & 0 
\end{array}\right)
\]
is the Conformal Discrete Fourier Transform associated to 
$$H_{\pi_1}=\begin{pmatrix}
    0 & h_1&0&0\\
    0&0&h_2&\\
    0&0&0&h_3\\
    h_4&0&0&0
\end{pmatrix}.$$
\end{example}

$\hfill{\diamond}$
\section{The conformal circulant matrix}

The following lemma provides a necessary and sufficient condition under which the blockwise summation of the powers of $H_{\pi_{s}}$ is possible.

\begin{proposition}
The nonzero blocks appearing in the powers of $H_{\pi_s}$ contain products with consecutive increasing indices,
\[
h_i h_{i+1}\cdots h_{i+k-1},
\]
for all $k=2,\ldots,n$ if and only if $s=1$.
\end{proposition}

\begin{proof}
Since each block row of $H_{\pi_s}$ contains exactly one nonzero block, namely $h_i$ which is located in the column $\pi_{s}(i)$, for which
\[
j\equiv i+s \pmod{n},
\] 
the multiplication by $H_{\pi_s}$ shifts the column index by $s \pmod {n}$. The nonzero contribution to an entry of $(H_{\pi_s})^k$ is
\[
\left(H_{\pi_s}\right)^k_{i,i+ks}
=
h_i\,h_{\pi_s(i)}\,h_{\pi_s^2(i)}\cdots h_{\pi_s^{\,k-1}(i)},
\]
where $\pi_s^t$ denotes the $t$-th iterate of $\pi_s$, and the factors appear in this precise order. We note that this order cannot be rearranged,
because the elements $h_i$ are non commutative.
Thus, these products have consecutive increasing indices if and only if
\[
\pi_{s}(i) =i+s\equiv i+1 \pmod{n},
\]
for every $i$. Hence
\[
s\equiv 1 \pmod{n}.
\]
Since $s\in\{1,2,\ldots,n-1\}$, this implies that $s=1.$

Conversely, if $s=1$, then $\pi_s(i)=i+1 \pmod n$, and therefore every nonzero block
product in $(H_{\pi_s})^k$ is

\begin{equation}\label{nonzeropowersH}
\left(H_{\pi_s}\right)^k_{i,i+k}
=
h_i h_{i+1}h_{i+2}\cdots h_{i+k-1},
\end{equation}
so, the products have consecutive increasing indices (recall that the indices are understood modulo $n$). Hence, the desired property occurs only in the case $s=1$.
\end{proof}
We now define the conformal circulant matrix.

\begin{definition}\label{conformal_x}
Let $h_i$ be matrices of order $m_i\times m_{i+1}$, $\textbf{x}= {(x_k)}\in\mathbb{K}^n$. Let $\mbox{\rm{Circ}}(\textbf{x})$ be the partitioned matrix such that $\mbox{\rm{Circ}}(\textbf{x})=\left[\mbox{\rm{Circ}}(\textbf{x})[m_i,m_{i+1}]\right] $ where, 
    \[
    \mbox{\rm{Circ}}(\textbf{x})[m_i,m_{i+1}]=\sum_{k=0}^{n-1} x_k (H_{\pi_1})^k.
    \]
\end{definition}

\begin{example}\label{example conformal circulant}
Consider the following conformal permutation matrix
\[
H_{\pi_1}=
\begin{pmatrix}
0 & h_1 & 0 & 0 & 0 & 0 & 0 & 0 & 0 & 0\\
0 & 0 & h_2 & 0 & 0 & 0 & 0 & 0 & 0 & 0\\
0 & 0 & 0 & h_3 & 0 & 0 & 0 & 0 & 0 & 0\\
0 & 0 & 0 & 0 & h_4 & 0 & 0 & 0 & 0 & 0\\
0 & 0 & 0 & 0 & 0 & h_5 & 0 & 0 & 0 & 0\\
0 & 0 & 0 & 0 & 0 & 0 & h_6 & 0 & 0 & 0\\
0 & 0 & 0 & 0 & 0 & 0 & 0 & h_7 & 0 & 0\\
0 & 0 & 0 & 0 & 0 & 0 & 0 & 0 & h_8 & 0\\
0 & 0 & 0 & 0 & 0 & 0 & 0 & 0 & 0 & h_9\\
h_{10} & 0 & 0 & 0 & 0 & 0 & 0 & 0 & 0 & 0
\end{pmatrix},
\]
where
\[
h_i\in M_{m_i\times m_{i+1}} (\mathbb C),
\quad i=1,\ldots,10,
\]
and the indices are taken $\pmod{n}$, that is,
$
h_{10}\in M_{m_{10}\times m_1}(\mathbb C).
$
In the next table we display some particular orders of the block matrices $h_i$:

\[
\begin{array}{|c|c|}
\hline
\text{$h_i$} & \text{size of the block $h_i$} \\ \hline
h_1 & 2\times 3\\
h_2 & 3\times 1\\
h_3 & 1\times 4\\
h_4 & 4\times 2\\
h_5 & 2\times 3\\
h_6 & 3\times 2\\
h_7 & 2\times 1\\
h_8 & 1\times 3\\
h_9 & 3\times 2\\
h_{10} & 2\times 2\\\hline
\end{array}
\]
\medskip

The nonzero entries of the powers of $H_{\pi_1}$ are given by \eqref{nonzeropowersH}, that is:
$$
(H_{\pi_1}^k)_{i,i+k}=
h_i h_{i+1}\cdots h_{i+k-1}, k=0, \ldots, 9, i=1, \ldots, 10.
$$

Note that each nonzero block of $H_{\pi_1}^k\in
M_{m_i\times m_j}(\mathbb{K}).$ Let $x = \bf{1}$. 
Therefore, the conformal circulant matrix
\[ 
\mbox{\rm{Circ}}({\bf 1})=
H_{\pi_1}^0+H_{\pi_1}^1+H_{\pi_1}^2+\cdots+H_{\pi_1}^9
\]
which is given by
\begin{equation*}
\makebox[\textwidth][c]{
  \resizebox{1.07\textwidth}{!}{
$\begin{pmatrix}
I_2 &
h_1 &
h_1h_2 &
h_1h_2h_3 &
h_1h_2h_3h_4 &
h_1h_2h_3h_4h_5 &
h_1h_2h_3h_4h_5h_6 &
h_1h_2h_3h_4h_5h_6h_7 &
h_1h_2h_3h_4h_5h_6h_7h_8 &
h_1h_2h_3h_4h_5h_6h_7h_8h_9
\\
h_2h_3h_4h_5h_6h_7h_8h_9h_{10} &
I_3 &
h_2 &
h_2h_3 &
h_2h_3h_4 &
h_2h_3h_4h_5 &
h_2h_3h_4h_5h_6 &
h_2h_3h_4h_5h_6h_7 &
h_2h_3h_4h_5h_6h_7h_8 &
h_2h_3h_4h_5h_6h_7h_8h_9
\\
h_3h_4h_5h_6h_7h_8h_9h_{10} &
h_3h_4h_5h_6h_7h_8h_9h_{10}h_1 &
I_2 &
h_3 &
h_3h_4 &
h_3h_4h_5 &
h_3h_4h_5h_6 &
h_3h_4h_5h_6h_7 &
h_3h_4h_5h_6h_7h_8 &
h_3h_4h_5h_6h_7h_8h_9
\\
h_4h_5h_6h_7h_8h_9h_{10} &
h_4h_5h_6h_7h_8h_9h_{10}h_1 &
h_4h_5h_6h_7h_8h_9h_{10}h_1h_2 &
I_4 &
h_4 &
h_4h_5 &
h_4h_5h_6 &
h_4h_5h_6h_7 &
h_4h_5h_6h_7h_8 &
h_4h_5h_6h_7h_8h_9
\\
h_5h_6h_7h_8h_9h_{10} &
h_5h_6h_7h_8h_9h_{10}h_1 &
h_5h_6h_7h_8h_9h_{10}h_1h_2 &
h_5h_6h_7h_8h_9h_{10}h_1h_2h_3 &
I_2 &
h_5 &
h_5h_6 &
h_5h_6h_7 &
h_5h_6h_7h_8 &
h_5h_6h_7h_8h_9
\\
h_6h_7h_8h_9h_{10} &
h_6h_7h_8h_9h_{10}h_1 &
h_6h_7h_8h_9h_{10}h_1h_2 &
h_6h_7h_8h_9h_{10}h_1h_2h_3 &
h_6h_7h_8h_9h_{10}h_1h_2h_3h_4 &
I_3 &
h_6 &
h_6h_7 &
h_6h_7h_8 &
h_6h_7h_8h_9
\\
h_7h_8h_9h_{10} &
h_7h_8h_9h_{10}h_1 &
h_7h_8h_9h_{10}h_1h_2 &
h_7h_8h_9h_{10}h_1h_2h_3 &
h_7h_8h_9h_{10}h_1h_2h_3h_4 &
h_7h_8h_9h_{10}h_1h_2h_3h_4h_5 &
I_2 &
h_7 &
h_7h_8 &
h_7h_8h_9
\\
h_8h_9h_{10} &
h_8h_9h_{10}h_1 &
h_8h_9h_{10}h_1h_2 &
h_8h_9h_{10}h_1h_2h_3 &
h_8h_9h_{10}h_1h_2h_3h_4 &
h_8h_9h_{10}h_1h_2h_3h_4h_5 &
h_8h_9h_{10}h_1h_2h_3h_4h_5h_6 &
I_2 &
h_8 &
h_8h_9
\\
h_9h_{10} &
h_9h_{10}h_1 &
h_9h_{10}h_1h_2 &
h_9h_{10}h_1h_2h_3 &
h_9h_{10}h_1h_2h_3h_4 &
h_9h_{10}h_1h_2h_3h_4h_5 &
h_9h_{10}h_1h_2h_3h_4h_5h_6 &
h_9h_{10}h_1h_2h_3h_4h_5h_6h_7 &
I_3 &
h_9
\\
h_{10} &
h_{10}h_1 &
h_{10}h_1h_2 &
h_{10}h_1h_2h_3 &
h_{10}h_1h_2h_3h_4 &
h_{10}h_1h_2h_3h_4h_5 &
h_{10}h_1h_2h_3h_4h_5h_6 &
h_{10}h_1h_2h_3h_4h_5h_6h_7 &
h_{10}h_1h_2h_3h_4h_5h_6h_7h_8 &
I_2
\end{pmatrix}
$}}
\end{equation*}
Thus, the block in position $(i,j)$ of the previous matrix is
\[
(\mbox{\rm{Circ}}(\mathbf{1}))_{ij}=
\begin{cases}
h_i h_{i+1}\cdots h_{j-1}, & i<j,\\
I_{m_i}, & i=j,\\
h_i h_{i+1}\cdots h_{10}h_1\cdots h_{j-1}, & i>j.
\end{cases}
\] 

or using the notation in the equations (\ref{gamma_n}), and (\ref{gamma_i}) we can rewrite: 
\[
(\mbox{\rm{Circ}}(\mathbf{1}))_{ij} = \begin{cases}  \gamma_{n-i} \, \gamma_{n-j}^{-1}, & i < j, \\  I_{m_i}, & i = j, \\  \gamma_{n-i} \, \gamma_{n-j}^{-1} \, \gamma^n, & i > j.  \end{cases}\quad\textit{  $i, j =0, 1, \ldots, 9$}
\]
\end{example}$\hfill{\diamond}$
\bigskip

\begin{theorem}\label{diag conf circulant} Let $\textbf{x}= {(x_k)}\in\mathbb{K}^n$ and $\mbox{\rm{Circ}}(\textbf{x})=[\mbox{\rm{Circ}}(\textbf{x})[m_i,m_{i+1}]]$ 
be a conformal circulant matrix as in Definition \ref{conformal_x}. Then, the spectrum of $\mbox{\rm{Circ}} (\textbf{x})$ contains the set
\begin{equation}\label{os mu_j's}
\bigcup_{j=0}^{n-1} \sigma \left(\boldsymbol{\mu}_j\right)    
\end{equation}
where $\boldsymbol{\mu}_j$ denote the $m_1\times m_1$ matrices $\sum_{k=0}^{n-1} x_k(\omega^{j}\gamma)^k$, for all $j=0,1,\ldots,n-1.$ Let $\gamma$ be an invertible and diagonalizable matrix of satisfying (\ref{gamma_n}) and $\gamma_i, i=0, \ldots, n$ satisfying (\ref{gamma_i}). If $\gamma_i$ are invertible for $i=1,\ldots,n$ and $m_1\cdot n=\sum\limits_{k=1}^n m_k$, then all eigenvalues of $\mbox{\rm{Circ}}(\textbf{x})$ are given by~\eqref{os mu_j's} where $\boldsymbol{\mu}_j$ are: 
$$
\small{
\begin{aligned}
 \frac{1}{n} \left[ n x_0 \, I_{m_1} + \sum_{0 \le i < s \le n-1} x_{s-i} \omega^{(s-r)j} I_{m_1} \;+\; \sum_{0 \le s < i \le n-1} x_{n+s-i} \, \omega^{(s-r)j} \, \left( \gamma_{n-s}^{-1} \, \gamma^n \, \gamma_{n-s} \right) \right].    
\end{aligned}
}
$$
\end{theorem}
\medskip
\begin{proof}
If $\gamma_i$ is invertible for $i=1,\ldots,n-1$, by Theorem~\ref{diagonalization H} it is clear that $\bigcup\limits_{j=0}^{n-1} \sigma \left(\boldsymbol{\mu}_j\right)$ contains the eigenvalues of $\mbox{\rm{Circ}} (\textbf{x})$. On the other hand, if  $m_1\cdot n=\sum\limits_{k=1}^n m_k$, we have, 
$$\mbox{\rm diag} (\mu_0, \ldots, \mu_{n-1}) = F_{\pi_1}^{-1} \mbox{\rm{Circ}}(\textbf{x})F_{\pi_1}$$ and 
\[
\mbox{\rm{Circ}}(\textbf{x})=F_{\pi_1} \diag\left(\sum_{k=0}^{n-1} x_k\gamma^k,\sum_{k=0}^{n-1} x_k(\omega\gamma)^k,\ldots,\sum_{k=0}^{n-1} x_k(\omega^{n-1}\gamma)^k\right)F^{-1}_{\pi_1}.
\]

Thus, $\sigma(\mbox{\rm{Circ}}(\textbf{x}))=\bigcup\limits_{j=0}^{n-1} \sigma \left(\boldsymbol{\mu}_j\right).$ 

\noindent In addition, as

\[
\diag(\boldsymbol{\mu}_0,\boldsymbol{\mu}_1,\ldots,\boldsymbol{\mu}_{n-1})=F^{-1}_{\pi_1}\mbox{\rm{Circ}}(\textbf{x})F_{\pi_1},
\]
then
\begin{align*}
    \boldsymbol{\mu}_j &= \sum_{r=0}^{n-1} \sum_{s=0}^{n-1} (F_{\pi_1}^{-1})_{j, r} \, (\mbox{\rm{Circ}}(\textbf{x}))_{r, s} \, (F_{\pi_1})_{s, j}\\
    &=\frac{1}{n} \sum_{r=0}^{n-1} \sum_{s=0}^{n-1}\gamma_{n-r}^{-1}(\mbox{\rm{Circ}}(\textbf{x}))_{r, s}\gamma_{n-s}\omega^{(s - r)j},
\end{align*}
where

\[
(\mbox{\rm{Circ}}(\textbf{x}))_{r, s}= 
\begin{cases}  
x_{s-r} \gamma_{n-r} \gamma_{n-s}^{-1}; & r < s, \\  
x_0; & r = s, \\ 
x_{n+s-r} \gamma_{n-r} \gamma_{n-s}^{-1} \gamma^n; & r > s.  
\end{cases}
\]
Decomposing the double sum for the cases $r < s$, $r = s$, and $r > s$ we obtain:

$$
\begin{aligned} \boldsymbol{\mu}_j = \frac{1}{n} \Biggr[ & \sum_{r=0}^{n-1} \gamma_{n-r}^{-1} \, x_0 \, \gamma_{n-r} \\ & + \sum_{0 \le r < s \le n-1} \gamma_{n-r}^{-1} \left( x_{s-r} \cdot \gamma_{n-r} \, \gamma_{n-s}^{-1} \right) \gamma_{n-s} \; \omega^{(s-r)j} \\ & + \sum_{0 \le s < r \le n-1} \gamma_{n-r}^{-1} \left( x_{n+s-r} \cdot \gamma_{n-r} \, \gamma_{n-s}^{-1} \, \gamma^n \right) \gamma_{n-s} \; \omega^{(s-r)j} \Biggr]. \end{aligned}$$

Thus, we can write:
$$
\begin{cases}
\gamma_{n-r}^{-1}x_0\gamma_{n-r}
=x_0I_{m_1}, 
& \mbox{\, if \,} r=s, \\[2mm]
\underbrace{\gamma_{n-r}^{-1}x_{r-s}\gamma_{n-r}}
_{=\,x_{r-s}I_{m_1}}
\gamma_{n-s}^{-1}\gamma_{n-s}
\omega^{(s-r)j}
=x_{r-s}\omega^{(s-r)k} I_{m_1},
& \mbox{\, if \,} r<s, \\[2mm]
\underbrace{\gamma_{n-r}^{-1}x_{n+s-r}\gamma_{n-r}}
_{=\,x_{n+s-r}I_{m_1}}
\gamma_{n-s}^{-1}\gamma^n\gamma_{n-s}
\omega^{(s-r)j}
=
x_{n+s-r}\gamma_{n-s}^{-1}\gamma^n
\gamma_{n-s}\omega^{(s-r)j},
& \mbox{\, if\, } r>s.
\end{cases}
$$

Then, using the previous expressions, $\boldsymbol{\mu}_j $ are given by
$$
\small{
\begin{aligned}
 \frac{1}{n} \left[ n x_0 \, I_{m_1} + \sum_{0 \le r < s \le n-1} x_{s-r} \omega^{(s-r)j} I_{m_1} \;+\; \sum_{0 \le s < r \le n-1} x_{n+s-r} \, \omega^{(s-r)j} \, \left( \gamma_{n-s}^{-1} \, \gamma^n \, \gamma_{n-s} \right) \right] ,   
\end{aligned}
}
$$
and the result follow.
\end{proof}

\begin{example}
   Let $H_{\pi_1}$ be as in Example~\ref{example conformal circulant}, where
   \begin{align*}
h_1 &= \begin{pmatrix} 1 & 0 & 0 \\ 0 & 1 & 0 \end{pmatrix}, &
h_2 &= \begin{pmatrix} 1 & 0 \\ 0 & 1 \\ 0 & 0 \end{pmatrix}, &
h_3 &= \begin{pmatrix} 1 & 0 & 0 & 0 \\ 0 & 1 & 0 & 0 \end{pmatrix} \\
h_4 &= \begin{pmatrix} 1 & 0 \\ 0 & 1 \\ 0 & 0 \\ 0 & 0 \end{pmatrix}, &
h_5 &= \begin{pmatrix} 1 & 0 & 0 \\ 0 & 1 & 0 \end{pmatrix}, &
h_6 &= \begin{pmatrix} 1 & 0 \\ 0 & 1 \\ 0 & 0 \end{pmatrix} \\
h_7 &= \begin{pmatrix} 1 & 0 \\ 0 & 1 \end{pmatrix}, &
h_8 &= \begin{pmatrix} 1 & 0 & 0 \\ 0 & 1 & 0 \end{pmatrix}, &
h_9 &= \begin{pmatrix} 1 & 0 \\ 0 & 1 \\ 0 & 0 \end{pmatrix} \\
h_{10} &= \begin{pmatrix} 1 & 0 \\ 0 & 1 \end{pmatrix}.
\end{align*}
In this case $\mbox{\rm{Circ}}(\mathbf{1})=(\mbox{\rm{Circ}}(\mathbf{1}))_{ij}$ where
\[
(\mbox{\rm{Circ}}(\mathbf{1}))_{i\neq j}= \begin{pmatrix} 1 & 0 & 0 & \dots & 0 \\ 0 & 1 & 0 & \dots & 0 \\ 0 & 0 & 0 & \dots & 0 \\ \vdots & \vdots & \vdots & \ddots & \vdots \\ 0 & 0 & 0 & \dots & 0 \end{pmatrix}=C_{i,j},
\]
where $C_{i,j}$ represents a block of order $d_i \times d_j.$ The dimensions of the blocks are: $d_1=2, d_2=3, d_3=2, d_4=4, d_5=2, d_6=3, d_7=2, d_8=2, d_9=3, d_{10}=2$. Thus, the matrix $\mbox{\rm{Circ}}(\mathbf{1})$ is as follows:

\[
\footnotesize{ 
\begin{pmatrix}
I_2 &C_{1,2}&C_{1,3}&C_{1,4}&C_{1,5}&C_{1,6}&C_{1,7}&C_{1,8}&C_{1,9}&C_{1,10} \\
C_{2,1}&I_3&C_{2,3}&C_{2,4}&C_{2,5}&C_{2,6}&C_{2,7}&C_{2,8}&C_{2,9}&C_{2,10}\\
C_{3,1}&C_{3,2}&I_2&C_{3,4}&C_{3,5}&C_{3,6}&C_{3,7}&C_{3,8}&C_{3,9}&C_{3,10}\\
C_{4,1}&C_{4,2}&C_{4,3}&I_4&C_{4,5}&C_{4,6}&C_{4,7}&C_{4,8}&C_{4,9}&C_{4,10}\\
C_{5,1}&C_{5,2}&C_{5,3}&C_{5,4}&I_2&C_{5,6}&C_{5,7}&C_{5,8}&C_{5,9}&C_{5,10}\\
C_{6,1}&C_{6,2}&C_{6,3}&C_{6,4}&C_{6,5}&I_3&C_{6,7}&C_{6,8}&C_{6,9}&C_{6,10}\\
C_{7,1}&C_{7,2}&C_{7,3}&C_{7,4}&C_{7,5}&C_{7,6} & I_2 & C_{7,8}&C_{7,9}&C_{7,10}\\
C_{8,1}&C_{8,2}&C_{8,3}&C_{8,4}&C_{8,5}&C_{8,6}&C_{8,7}&I_2&C_{8,9}&C_{8,10}\\
C_{9,1}&C_{9,2}&C_{9,3}&C_{9,4}&C_{9,5}&C_{9,6}&C_{9,7}&C_{9,8}&I_3&C_{9,10}\\
C_{10,1}&C_{10,2}&C_{10,3}&C_{10,4}&C_{10,5}&C_{10,6}&C_{10,7}&C_{10,8}&C_{10,9} & I_2
\end{pmatrix},
}
\]

and $\gamma=\begin{pmatrix} 1 & 0\\ 0 & 1 \end{pmatrix}$. 

By Theorem~\ref{diag conf circulant} $\mbox{\rm{Circ}}(\mathbf{1})$ has eigenvalues $\bigcup\limits_{j=0}^{9} \sigma \left(\boldsymbol{\mu}_j\right)$, where

\begin{eqnarray*}
\boldsymbol{\mu}_0&=&\sum_{k=0}^9\gamma^k=\begin{pmatrix} 10 & 0 \\ 0 & 10 \end{pmatrix};\\
\quad\boldsymbol{\mu}_j &=& \sum_{k=1}^9\omega^{jk}\gamma^k=\left(\sum_{k=0}^9\omega^{jk}\right)\begin{pmatrix} 1 & 0 \\ 0 & 1 \end{pmatrix}\\
&=&\dfrac{1-\omega^{10j}}{1-\omega^{j}}\begin{pmatrix} 1 & 0 \\ 0 & 1 \end{pmatrix}=\begin{pmatrix} 0 & 0 \\ 0 & 0 \end{pmatrix}, \mbox{\,with\,\,} j=1, \ldots, 9.
\end{eqnarray*}

Thus,
\[
\boldsymbol{\mu}_0=\begin{pmatrix} 10 & 0 \\ 0 & 10 \end{pmatrix}\quad\textit{and}\quad\boldsymbol{\mu}_1 = \boldsymbol{\mu}_2 =\cdots =\boldsymbol{\mu}_9 = \begin{pmatrix} 0 & 0 \\ 0 & 0 \end{pmatrix}.
\]
Therefore, twenty of the eigenvalues of $\mbox{\rm{Circ}}(\mathbf{1})$ are 
\begin{align*}
    \lambda=10&\quad \textit{ with algebraic multiplicity 2,}\\
    \lambda=0&\quad \textit{ with algebraic multiplicity 18.}
\end{align*}
On the other hand, if we compute all the eigenvalues of $\mbox{\rm{Circ}}(\mathbf{1})$ we obtain
\begin{align*}
   \lambda=10&\quad \textit{ with algebraic multiplicity 2,}\\
    \lambda=0&\quad \textit{ with algebraic multiplicity 18}\\
    \lambda=1&\quad \textit{ with algebraic multiplicity 5.}
\end{align*}
\end{example}$\hfill{\diamond}$\\
Note that Theorem~\ref{diag conf circulant} allows us to explicitly calculate $m_1\cdot n$ eigenvalues of $\mbox{\rm{Circ}}(\mathbf{x})$.\\
 \begin{theorem}
  Let $\textbf{x}= {(x_k)}\in\mathbb{K}^n$ and $\mbox{\rm{Circ}}(\textbf{x})=[\mbox{\rm{Circ}}(\textbf{x})[m_i,m_{i+1}]]$ be a conformal circulant matrix as in Definition \ref{conformal_x}. Then
    \begin{equation*}
        x_k=\frac{1}{n} \left( \sum_{j=0}^{n-1} \omega^{-jk} \, \gamma_n \, \boldsymbol{\mu}_j \right) \gamma_n^{-1},\quad\textit{ fo all } k = 0, 1, \dots, n-1,
    \end{equation*} 
    provided that $\gamma_i$ is invertible for all $i=1, \ldots, n$ and $m_1\cdot n=\sum\limits_{k=1}^n m_k$. 
\end{theorem}

\begin{proof} The $(r,s)$ entries of the matrix $\mbox{\rm{Circ}}(\textbf{x})$ are:
\[
(\mbox{\rm{Circ}}(\textbf{x}))_{r,s}=(F_{\pi_1}\diag(\boldsymbol{\mu}_0,\boldsymbol{\mu}_1,\ldots,\boldsymbol{\mu}_{n-1})F^{-1}_{\pi_1})_{r,s}
\]
 where, $F_{\pi_1}=\diag(\gamma_n,\gamma_{n-1},\ldots,\gamma_1) \circledc F$, with $F$ the discrete Fourier transform. Then,
 \[
 (\mbox{\rm{Circ}}(\textbf{x}))_{0,s} = \begin{cases}  x_0 I_{m_1}, & \mbox{\, if \,} s = 0, \\ x_{s} \cdot \gamma_n \, \gamma_{n-s}^{-1}, & \mbox{\, if\, }1 \le s \le n-1, \end{cases}
 \]
 and 
 \[ (F_{\pi_1}\diag(\boldsymbol{\mu}_0,\boldsymbol{\mu}_1,\ldots,\boldsymbol{\mu}_{n-1})F^{-1}_{\pi_1})_{0,s} = \frac{1}{n} \left( \sum_{k=0}^{n-1} \omega^{-ks} \, \gamma_n \, \boldsymbol{\mu}_k \right) \gamma_{n-s}^{-1}.
 \]
 Then,
 \[
 x_s=
 \begin{cases} 
 \frac{1}{n} \, \gamma_n \left( \sum\limits_{k=0}^{n-1} \boldsymbol{\mu}_k \right) \gamma_n^{-1}; & \mbox{\, if \,} s = 0, \\ 
 \frac{1}{n} \left( \sum\limits_{k=1}^{n} \omega^{-ks} \, \gamma_n \, \boldsymbol{\mu}_k \right) \gamma_n^{-1}, & \mbox{\, if \,} 1 \le s \le n-1. 
 \end{cases}
 \]
 Therefore, 
 $$x_s = \frac{1}{n} \left( \sum_{j=0}^{n-1} \omega^{-js} \, \gamma_n \, \boldsymbol{\mu}_j \right) \gamma_n^{-1},\quad \mbox {for all } s = 0, 1, \dots, n-1$$
 \end{proof}   

\begin{theorem}\label{characterization structure}
    A rectangular matrix $A$ is a conformal circulant matrix if and
only if
\[
A=\diag(\gamma_n,\gamma_{n-1},\ldots,\gamma_1)\circledc C\circledc(\diag(\gamma_n,\gamma_{n-1},\ldots,\gamma_1))^{-1}
\]
for some block circulant matrix $C$ and some rectangular matrices $h_{j}\in M_{m_j\times m_{j+1}}(\mathbb{K}), $ with $j=1, \ldots, n$ and $m_1\cdot n=\sum\limits_{k=1}^n m_k$. 
\end{theorem}

\begin{proof}
    Let $\textbf{x}=(x_0,x_1,\ldots,x_{n-1})\in\mathbb{K}^n$, $\{h_{i}\}$ be a family of $m_i\times m_{i+1}, i=1,\ldots, n$ rectangular matrices with $m_{n+1}=m_1$, and $\gamma$ be an invertible and diagonalizable matrix satisfying
\begin{equation}
\gamma^n=h_1h_2\cdots h_n.
\end{equation} The matrix $A$ is a conformal circulant matrix if only if,
    \[
    F_{\pi_1}^{-1}AF_{\pi_1}=\sum_{k=0}^{n-1}x_k(\diag(\gamma,\omega\gamma,\ldots,\omega^{n-1}\gamma))^k,
    \]
for some $\bf x\in\mathbb{K}^n$ and some rectangular matrices $h_j$ in the previous conditions. Thus,
\[
F_{\pi_1}^{-1}AF_{\pi_1}=\diag(\boldsymbol{\mu}_0,\boldsymbol{\mu}_1,\ldots,\boldsymbol{\mu}_{n-1}).
\]
Then
\small{
\begin{align*}
A&=F_{\pi_1}\diag(\boldsymbol{\mu}_0,\boldsymbol{\mu}_1,\ldots,\boldsymbol{\mu}_{n-1})F_{\pi_1}^{-1}\\
&=\diag(\gamma_n,\gamma_{n-1},\ldots,\gamma_1)\circledc F\diag(\boldsymbol{\mu}_0,\boldsymbol{\mu}_1,\ldots,\boldsymbol{\mu}_{n-1}) (\diag(\gamma_n,\gamma_{n-1},\ldots,\gamma_1)\circledc F)^{-1}\\
&=\diag(\gamma_n,\gamma_{n-1},\ldots,\gamma_1)\circledc F\diag(\boldsymbol{\mu}_0,\boldsymbol{\mu}_1,\ldots,\boldsymbol{\mu}_{n-1})F^{-1}\circledc(\diag(\gamma_n,\gamma_{n-1},\ldots,\gamma_1))^{-1}\\
&=\diag(\gamma_n,\gamma_{n-1},\ldots,\gamma_1)\circledc C\circledc(\diag(\gamma_n,\gamma_{n-1}\ldots,\gamma_1))^{-1}
\end{align*}
}
 where $\gamma_i=h_{n-i+1}\gamma_{i-1}\gamma^{-1}$, for $i=1,\ldots,n$, with $\gamma_0=I_{m_1}$ and the matrix $$C=F\diag(\boldsymbol{\mu}_0,\boldsymbol{\mu}_1,\ldots,\boldsymbol{\mu}_{n-1})F^{-1},$$ is some block circulant matrix.
\end{proof}
The next corollary can be easily shown using the characterization in Theorem  \ref{characterization structure} and the properties of block circulant matrices. Provided that the order that defines them allows the indicated matrix operations to be performed.
\begin{coro}
     If $A=\mbox{\rm{Circ}}(\textbf{x})$ and $B=\mbox{\rm{Circ}}(\textbf{y})$ are conformal circulant matrices for $\textbf{x}, \textbf{y} \in\mathbb{K}^n$ defined by same $\gamma$ and $\gamma_k$ rectangular matrices, and $\alpha$ and $\beta$ are scalars, then $A^{-1}$ $A^T$, $\alpha A+\beta B$ and $AB$ are conformal circulant matrices.
\end{coro}

\section{Conclusions}

In this work, the conformal matrices $H_{\pi_1}$ corresponding to weighted shift operators (a linear operator that maps each vector of a fixed basis to a scalar multiple of another basis vector, with the scalar representing the corresponding weight)
defined on the vector space $\mathbb{C}^N$ and associated with the weighted adjacency matrix of a directed cyclic graph $C_N$, are defined and studied. In advanced linear algebra, this family of operators is characterized by its non-normality ($H_{\pi_1} H_{\pi_1}^\dagger \neq H_{\pi_1}^\dagger H_{\pi_1}$) (where $H_{\pi_1}^\dagger$ denotes the conjugate transpose (adjoint) of $H_{\pi_1}$), and its spectral analysis is driven by the study of systems with transport asymmetry \cite{Hatano}. In quantum mechanics and neutrino phenomenology, this algebraic structure acquires direct physical meaning when the effective Hamiltonian of the system $H_{\text{eff}} \in \mathbb{C}^{k \times k}$ (an operator that describes the dynamics of an effective system) ceases to be self-adjoint (Hermitian) in order to describe dissipative dynamics, decays, or asymmetric interactions with the environment \cite{Berryman,Mena}.The analysis of the flavor transition matrix in neutrino oscillations—modeled as a system of $k$ flavor states—depends critically on the non-orthogonal spectral decomposition of $H_{\text{eff}}$ and its left and right eigenvectors (biorthogonal bases) \cite{Berryman,Hatano,Sunderhauf}. Furthermore, it is of particular interest to consider matrix polynomials of the form $\sum_{i=0}^{n-1} x_i H_{\pi_1}^i$, since they constitute a natural and physically relevant generalization of classical circulating matrix structures.
On the other hand, the matrix $F_{\pi_1}$ can be viewed as a generalization of the Discrete Fourier transform matrix, as its columns consist of the eigenvectors of $\mbox{\rm{Circ}}(\textbf{x})$. However, key properties, such as unitarity ($F_{\pi_1}F_{\pi_1}^\dagger = I$), are not preserved because the eigenvectors of $\mbox{\rm{Circ}}(\textbf{x})$ are not mutually orthogonal. A detailed analysis of the properties and applications of $F_{\pi_1}$ remains a topic for future research.

\textbf{Acknowledgments}

E.~Andrade is supported by CIDMA (\url{https://ror.org/05pm2mw36}) under the Portuguese Foundation for Science and Technology (FCT, \url{https://ror.org/00snfqn58}), through the grants\\
\href{https://doi.org/10.54499/UID/04106/2025}{UID/04106/2025} and
\href{https://doi.org/10.54499/UID/PRR/04106/2025}{UID/PRR/04106/2025}.

\bibliographystyle{alpha}
\bibliography{references}
\end{document}